\documentclass{amsart}

\usepackage{amssymb}
\usepackage{amsmath}
\usepackage{bbm, dsfont}
\usepackage{graphics}
\usepackage{mathrsfs}
\usepackage{color,cite,graphicx}
\definecolor{gray}{gray}{0.4}
\usepackage[all]{xy}
\usepackage{enumerate}

\title[Order sixteen]{Classification of order sixteen non-symplectic automorphisms on K3 surfaces}

\author{Dima Al Tabbaa}
\address{Laboratoire de Math\'ematiques et Applications, UMR CNRS 6086,
			Universit\'e de Poitiers, T\'el\'eport 2, Boulevard Marie et Pierre Curie,
			86962 FUTUROSCOPE CHASSENEUIL, France}
\email{Dima.Al.Tabbaa@math.univ-poitiers.fr}

\author{Alessandra Sarti}
\address{Laboratoire de Math\'ematiques et Applications, UMR CNRS 6086,
			Universit\'e de Poitiers, T\'el\'eport 2, Boulevard Marie et Pierre Curie,
			86962 FUTUROSCOPE CHASSENEUIL, France}
\email{sarti@math.univ-poitiers.fr}
\urladdr{http://www-math.sp2mi.univ-poitiers.fr/~sarti/}

\author{Shingo Taki}
\address{School of Information Environment, Tokyo Denki University,
2-1200 Muzai Gakuendai, Inzai-shi, Chiba 270-1382, Japan}
\email{staki@mail.dendai.ac.jp}
\urladdr{http://www.math.sie.dendai.ac.jp/~taki/}

\date{\today}

\newtheorem{lemma}{Lemma}
\newtheorem{pro}{Proposition}

\newtheorem{theorem}{Theorem}[section]
\theoremstyle{definition}

\newtheorem{example}[theorem]{Example}

\newtheorem{remark}[theorem]{Remark}

\DeclareMathOperator{\Pic}{Pic}

\DeclareMathOperator{\rank}{rk}

\DeclareMathOperator{\Fix}{Fix}

\DeclareMathOperator{\rk}{rk}

\newcommand{\IC}{\mathbb{C}}

\newcommand{\IZ}{\mathbb{Z}}

\newcommand{\IP}{\mathbb{P}}


\thanks{The third-named author was partially supported by 
Research Institute for Science and Technology of Tokyo Denki University 
Grant Number Q14K-06/Japan.}

\begin{document}

\begin{abstract}
In the paper we classify K3 surfaces with non-symplectic automorphism of order $16$
in full generality.
We show that 
the fixed locus contains only rational curves and points and we completely classify the seven possible
configurations.  If the N\'eron-Severi group has rank 6, there are two possibilities and 
 if its rank is 14, there are five possibilities. In particular if the action
of the automorphism is trivial on the N\'eron-Severi group, then we show that its rank is six.

\end{abstract}

\maketitle

\section*{Introduction}
Automorphisms of K3 surfaces were widely studied in the last years,  in particular also for the recent
 relation with the Bloch conjecture, see e.g. \cite{Huybrechts}, \cite{FU}. Here we study (purely) non-symplectic automorphisms of order $d$,
i.e. automorphisms that multiply the non degenerate holomorphic two form by a primitive $d$--root of the unity. 
The study of non-symplectic automorphism of prime order was completed by Nikulin in \cite{nikulinfactor}
in the case of involutions, and more recently by Artebani, Sarti and Taki in several papers \cite{AS3, ast, takiauto}
for the other prime orders. The study of non-symplectic automorphisms of not prime order turn out to
be more complicated, in fact in this situation the ''generic'' case does not imply that the action of the automorphism 
 is trivial on the N\'eron-Severi lattice. In the paper \cite{Taki} Taki completely describes the case 
when the action is trivial on the N\'eron-Severi lattice and the automorphism is a prime power. If we consider non-symplectic automorphisms that are of order $2^t$, then by results of Nikulin we have $0\leq t\leq 5$, and 
by a recent paper by Taki \cite{Taki32} there is only one K3 surface that admits an order 32 non-symplectic automorphism. 
Some further results in this direction are contained in a paper by Sch\"utt \cite{matthias} in the case of automorphisms
of a $2$-power order and in a paper by Artebani and Sarti \cite{artsa}, in the case of the order 4. In this last paper 
the hypothesis of trivial action on 
the N\'eron-Severi lattice is left out. Here we 
consider 
the case of the order $16$ in all generality, which together with the order $8$ 
remained quite unexplored. 

Since the Euler function of $16$ divides the rank of the transcendental lattice
(see \cite{Nikulin1}) the rank of the N\'eron-Severi group can be only 6 or 14. 
More precisely let $X$ be a K3 surface, $\omega_X$ a generator of $H^{2,0}(X)$, $\sigma$ an order 
16 automorphism such that $\sigma^* \omega_X=\zeta_{16} \omega_X$, where $\zeta_{16}$ denotes a primitive order 16 root of unity. 
We first show that if the fixed locus of $\sigma$ contains a curve then its genus is zero (Proposition \ref{ratio}), 
then in the case that $\rk\Pic(X)=6$ we have the following number of isolated fixed points $N$ and fixed rational curves $k$, (Theorem \ref{rank6}):
$$
(\Pic(X), N, k)=(U\oplus D_4,6,1),~\mbox{or} \qquad (U(2)\oplus D_4, 4,0). 
$$
 In the first case the action is trivial on $\Pic(X)$ but not in the second case.
If $\rk\Pic(X)=14$ and $\sigma^4$ fixes an elliptic curve $C$,  
then $\sigma$ preserves $C$ and the induced $\sigma$-invariant elliptic fibration induced has a reducible fiber of type $IV^*$
and the number of isolated fixed points and fixed rational curves are as follows: $(N,k)=(8,1)$ or $(6,0)$. In the first case $\sigma$  preserves each component of the fiber $IV^*$ and in the second case it acts as a reflection on it. In any case the action is not trivial on $\Pic(X)$, (Proposition \ref{elliptic}). Finally
if $\rk \Pic(X)=14$ and if $\Fix(\sigma^4)$ contains a curve of genus bigger than 1 we have the three cases with 
$(\Pic(X),N,k)$ equal to:
$$
(U\oplus D_4\oplus E_8,12,1),\qquad (U(2)\oplus D_4\oplus E_8, 4, 0)~\mbox{or} \qquad (U(2)\oplus D_4\oplus E_8, 10, 1).
$$

In these three cases the action of $\sigma$ is not trivial on $\Pic(X)$, (Theorem \ref{rank14}). 
This in particular shows that there does not exist a K3 surface $X$ with Picard number $14$ with an automorphism of order 16 acting non symplectically on it and trivially on $\Pic(X)$. This corrects a small mistake in the paper \cite{Taki},
where the author claims that such a K3 surface exists. 

We construct the K3 surfaces
in the Examples \ref{fibr0}, \ref {fibr1}, \ref{fibr2}, except in the case of $\Pic(X)=U(2)\oplus D_4\oplus E_8$, 
and $(N,k)=(10,1)$ which we do not know if it exists. For the proofs of the Theorems \ref{elliptic}, \ref{rank6}, \ref{rank14},  we use Lefschetz formulas, the results on non-symplectic involutions and on non-symplectic order four automorphisms are contained in \cite{artsa}, \cite{Taki}. We use also results on non-symplectic automorphisms of order eight.
The results of this paper are partially contained in the forthcoming PhD thesis of the first author
under the supervision of the second author. The results of the paper on order eight non-symplectic automorphism as
 well as a classification of K3 surfaces with non-symplectic automorphism of order eight 
 will be contained in the PhD thesis of Al Tabbaa, \cite{dimathesis} too. 
  
 
{\em Acknowledgements: We warmly thank Michela Artebani, Samuel Boissi\`ere and Alice Garbagnati for several interesting discussions.}  
\section{The fixed locus}

Let $X$ be a K3 surface with a {\it non-symplectic} automorphism $\sigma$ of order 16, 
this means that the action of $\sigma^*$ on the vector space $H^{2,0}(X)=\IC \omega_X$ of holomorphic two-forms is not trivial. 
More precisely we assume that $\sigma^*\omega_X=\zeta_{16} \omega_X$, where 
$\zeta_{16}$ is a primitive root of the unity of order 16 (this action is called sometime in the literature {\it purely non-symplectic}). 
For simplicity we denote by $\zeta:=\zeta_{16}$ and by $\xi:=\zeta_{16}^2$ which is a primitive root of the unity of order $8$. 

 We denote furthermore by $r_{\sigma^i},l_{\sigma^i}, m_{\sigma^i}$, $m^1_{\sigma^i}$, $m^2_{\sigma^i}$, $i=1,2,4,8$ 
the rank of the eigenspace of $(\sigma^i)^*$ in $H^2(X,\IC)$ relative to the eigenvalues $1,-1$, $i$, $\xi$ and $\zeta$.  
For simplicity for $i=1$ we just write $r_{\sigma}$, $l_{\sigma},\ldots$ or even $r, l,\ldots$. 
The following relations holds:
\begin{eqnarray}\label{ranks}
\begin{array}{cccc}
r_{\sigma^2}=r_\sigma+l_\sigma,& l_{\sigma^2}=2m_{\sigma}& m_{\sigma^2}=2m^1_\sigma, &
m^1_{\sigma^2}=2m^2_\sigma\\ 
r_{\sigma^4}=r_\sigma+l_\sigma+2m_\sigma& l_{\sigma^4}=4m^1_\sigma&m_{\sigma^4}=4m^2_\sigma&\\
r_{\sigma^8}=r_\sigma+l_\sigma+2m_\sigma+4m^1_\sigma& l_{\sigma^8}=8m^2_\sigma&&\\
r_{\sigma}+l_{\sigma}+2m_{\sigma}+4m^1_{\sigma}+8m^2_{\sigma}=22&&&\\
\end{array}
\end{eqnarray}

 Moreover, let
$$
S(\sigma^i)=\{x\in H^2(X,\IZ)\,|\, (\sigma^i)^*(x)=x\},
$$
$$
T(\sigma^i)=S(\sigma^i)^{\perp}\cap H^2(X,\IZ).
$$ 
Observe that in the generic case we can assume that $\Pic(X)=S(\sigma^8)$, i.e. the action of the involution $\sigma^8$ is trivial on $\Pic(X)$. We have moreover that $S(\sigma)\subset \Pic(X)$ and so the transcendental lattice satisfies 
$T_X\subset T(\sigma)$. Since the action of $\sigma$ on  $T_X$ and $T(\sigma)$ is by primitive roots of the unity, see
\cite{Nikulin1}, we have $\rk(T_X)=8m^2_{\sigma}$. Since $\rk(T_X)\leq 21$ we have in fact only two possibilities which are $m^2_{\sigma}=1$ or $2$ so that $\rk S(\sigma)=14$ respectively $6$. 
Observe moreover that $r_\sigma>0$ since there is always an ample invariant class on $X$ (see \cite[Theorem 3.1]{Nikulin1}). 

We start recalling the following result about non-symplectic involutions (see \cite[Theorem 4.2.2]{nikulinfactor}).
\begin{theorem}\label{inv}
Let $\tau$ be a non-symplectic involution on a K3 surface $X$. The fixed locus of $\tau$ is 
either empty, the disjoint union of two elliptic curves or the disjoint union of a 
smooth curve of genus $g\geq 0$ and 
$k$ smooth rational curves. 

Moreover, its fixed lattice $S(\tau)\subset \Pic(X)$ is a $2$-elementary lattice with determinant $2^a$ such 
that:
\begin{itemize}
\item $S(\tau)\cong U(2)\oplus E_8(2)$ iff the fixed locus of $\tau$ is empty;
\item  $S(\tau)\cong U\oplus E_8(2)$ iff $\tau$ fixes two elliptic curves;
\item $2 g=22-\rk S(\tau)-a$ and $2k=\rank S(\tau)-a$ otherwise.
\end{itemize}  
\end{theorem}
Recall that at a fixed point for $\sigma^i$ the action can be linearized and is given by a matrix 
as
$$
A^i_{j,k}= \left(\begin{array}{cc}
\zeta^j_{(16/i)} &0\\
0& \zeta^k_{(16/i)}
\end{array}
\right)
$$
with $j+k=1\mod (16/i)$. This means that the fixed locus of $\sigma^i$ is the disjoint union of smooth curves and isolated points. We denote by $N_{\sigma^i}$ respectively by $k_{\sigma^i}$ the fixed points and fixed rational
curves in  $\Fix(\sigma^i)$. Moreover by $n^{\sigma^i}_{j,k}$ we denote the number of isolated fixed points of type $(j,k)$ by 
$\sigma^i$. In several cases when it is clear which automorphism are we considering we just write $n_{j,k}$.

\begin{lemma}\label{exchanged}
Let $A$ be the number of pairs of rational curves interchanged by $\sigma^4$ and fixed by $\sigma^8$,
then $A\in 4\IZ$.
\end{lemma}
\begin{proof}
A curve as in the statement has stabilizer group in $\langle \sigma\rangle$ of order 2.
Hence its $\sigma$-orbit has length 8, so we get that $A$ is a multiple of $4$. 
\end{proof}

\begin{pro}\label{genus}
Let $\sigma$ be a purely non-symplectic automorphism of order sixteen on a K3 surface $X$. 
Then if $C\subset \Fix(\sigma)$ we have $g(C)\leq 1$. 
\end{pro}
\begin{proof}
If $C\subset\Fix(\sigma)$ with $g(C)\geq 2 $ then this is also fixed by 
$\sigma^4$ which is non-symplectic of order 4. By the relations \eqref{ranks} we have that $l_{\sigma^4}$ and 
$m_{\sigma^4}$ are multiples of $4$. Then checking in \cite[Theorem 4.1]{artsa} the only possible case 
is $(m_{\sigma^4},r_{\sigma^4},l_{\sigma^4})=(4,6,8)$ and $N_{\sigma^4}=2$, $k_{\sigma^4}=0$, $g(C)=2$. 
Moreover there are $4$ curves interchanged two by two  by $\sigma^4$ so $A=2$ contradicting Lemma \ref{exchanged}
\end{proof}
Recall also the following useful Lemma, see e.g. \cite[Lemma 4]{artsa}:
\begin{lemma}\label{tree}
Let $T=\sum_i R_i$ be a tree of smooth rational curves on a K3 surface $X$ such that each $R_i$ 
is invariant under the action of a purely non-symplectic automorphism $\sigma$ of order $k$. 
Then, the points of intersection of the rational curves $R_i$ are fixed by $\sigma$ and the action 
at one fixed point determines the action on the whole tree. 
\end{lemma}
\begin{remark}\label{suite}
In the case of an automorphism of order $16$, with the assumption of Lemma \ref{tree}, 
the local actions at the intersection points of the curves $R_i$ appear in the following order
(we give only the exponents of $\zeta$ in the matrix of the local action):
\begin{eqnarray*}
\ldots,(0,1), (15,2),(14,3),(13,4),(12,5),(11,6),(10,7),(9,8),\\
(8,9),(7,10),(6,11),(5,12),(4,13),(3,14),(2,15),(1,0),\ldots
\end{eqnarray*}
\end{remark}
\begin{pro}\label{lefschetz16}
Let $\sigma$ be a purely non-symplectic automorphism of order $16$ acting on a $K3$ surface $X$. Then
the fixed locus is non empty and
$$
\Fix(\sigma)=C\cup E_1 \cup\cdots\cup E_k\cup \{ p_1,\cdots, p_N \} .
$$
where $C$ is a curve of genus $g\geq 0$, the $E_i$ are rational fixed curves, $k=k_{\sigma}$ and the $p_i$ are isolated fixed points, $N=N_{\sigma}$.
Moreover $N$ is even, $4\leq N\leq 16$ and the following relations hold :
\[ \label{N1}
\tag{I}
N=n_{3,14}+n_{4,13}+n_{5,12}+n_{6,11}+2n_{7,10}+2k+1 .
\]
\[ \label{N2}
\tag{II}
N=2n_{3,14}+2n_{5,12}+2n_{7,10}+2k .
\]
\[ \label{N3}
\tag{III}
N=2+r_{\sigma}-l_{\sigma}-2k .
\]
\end{pro}
\begin{proof}
By Proposition \ref{genus} we know that $g(C)=0$ or $g(C)=1$. 
We use first the topological Lefschetz fixed point formula for $\sigma$. We write $r=r_\sigma$ and $l=l_\sigma$.
This gives $N+2k=\chi(\Fix(\sigma))=r-l+2$ so $r-l=N+2k-2$. Since $\rk S(\sigma)=14$ or $6$
in any case we have $N\leq 16$. We have that the Lefschetz number is $L(\sigma)=1+\zeta^{-1}$ so 
using Lefschetz formula we obtain the equations:
\begin{equation}\label{equ1}
 n_{2,15} - n_{7,10} + n_{8,9} =1+ 2k.
\end{equation}
\begin{equation}\label{equ2}
 n_{2,15} - n_{3,14} + n_{4,13} - n_{5,12} + n_{6,11} - n_{7,10} + n_{8,9} =2k .
\end{equation}
\begin{equation}\label{equ3}
n_{4,13} +n_{5,12} - 2n_{6,11} +2n_{7,10} - n_{8,9} = 2k.
\end{equation}
\begin{equation}\label{equ4}
 2n_{3,14} - 2n_{4,13} + 2n_{6,11} - n_{8,9}= 2k.
 \end{equation}
and combining \eqref{equ1}  and \eqref{equ2} we get
\begin{equation}\label{equ5}
 n_{3,14} - n_{4,13} + n_{5,12} -n_{6,11} =1.
\end{equation}

From \eqref{equ1} and \eqref{equ2} and the fact that $N=\sum n_{j,k}$ we obtain the relations \eqref{N1} and \eqref{N2} in the statement respectively. By \eqref{N1} we get that $N\geq1$ and by \eqref{N2} we find that $N$ is an even number,
thus $N\geq 2$. If $N=2$ then by \eqref{N1} we obtain $k=n_{7,10}=0$ and either $n_{3,14}$ or $n_{5,12}$ equal to 1 by relations \eqref{N1} and \eqref{N2} , thus $n_{4,13}=n_{6,11}=0$ by \eqref{N1} 
and either $n_{2,15}$ or $n_{8,9}=1$ by \eqref{equ1}. By \eqref{equ4} we obtain $n_{8,9}=2n_{3,14}$ so
$n_{8,9}=n_{3,14}=0$. By using \eqref{equ3} we obtain $n_{5,12}=0$ which is impossible. So $N\geq 4$. 
\end{proof}
\begin{remark}\label{nn} 
\begin{itemize}
\item[1)] As a direct consequence of formulas in Proposition \ref{lefschetz16} we find that if $N=4$ we have only the 
possibility with $(n_{3,14},n_{7,10},n_{8,9},k)=(1,1,2,0)$ (the other $n_{i,j}$ are zero) so that $r-l=2$.

The case $(N,k)=(8,0)$ is not possible.

If $(N,k)=(6,0)$ then $(n_{5,12},n_{6,11},n_{7,10},n_{8,9})=(2,1,1,2)$ the other $n_{i,j}$ are zero.

If $(N,k)=(6,1)$ then 
$(n_{2,15},n_{3,14},n_{7,10})=(4,1,1)$ the other $n_{i,j}$ are zero.
\item[2)] The fixed points for $\sigma$ with local action $(2,15)$, $(7,10)$, $(3,14)$, $(6,11)$, are isolated fixed points for $\sigma^4$, whence the points of type $(8,9)$, $(4,13)$ and $(5,12)$ are contained on a fixed curve for $\sigma^4$. Finally the points of type $(8,9)$ are contained on a fixed curve for $\sigma^2$. 
\end{itemize}
\end{remark}
\begin{pro}\label{ratio}
If $C\subset \Fix(\sigma)$ then $C$ is rational.
\end{pro}
\begin{proof}
By \cite[Theorem 3.1]{artsa} if $g(C)=1$ and since by formulas \eqref{ranks} we have $l_{\sigma^4}, m_{\sigma^4}\in 4\IZ$, we get 
$(m_{\sigma^4},r_{\sigma^4},l_{\sigma^4})=(4,10,4)$ and the fixed locus of
$\sigma^4$ contains 1 rational fixed curve and 6 isolated fixed points (here $A=0$). Observe moreover 
that since $C\subset \Fix(\sigma)$ also $\sigma$ preserves the elliptic fibration determined by $C$. The automorphism
$\sigma^4$ acts with order four on the basis of the fibration by \cite[Theorem 3.1]{artsa} so $\sigma$ acts with order $16$ on it and fixes two points. 
One point corresponds to the smooth elliptic curve $C$ the other point to the 
fiber $IV^*$. The component of multiplicity $3$ in the fiber $IV^*$ is clearly $\sigma$-invariant. 
If it is fixed by $\sigma$ then each other component is preserved, so that $k=1$ and $N=6$. 
More precisely 
by Remark \ref{suite} we have $n_{2,15}=n_{3,14}=3$ which contradicts Remark \ref{nn}. If the component of multiplicity $3$ is $\sigma$-invariant then it contains $2$ isolated fixed points. Two branches of the fiber are exchanged
and we have $N=4$. By Remark \ref{nn} we have $n_{8,9}=2$, $n_{7,10}=1$, $n_{3,14}=1$ but this is not possible
by using the Remark \ref{suite}. 
\end{proof}
\begin{pro}\label{fix4}
The fixed locus $\Fix(\sigma^4)$ contains at least a fixed curve and we have $g(C)\leq 1$.
\end{pro}
\begin{proof}
If $\Fix(\sigma^4)$ contains only isolated fixed points then by Remark \ref{nn} we have $n_{4,13}=n_{5,12}=n_{8,9}=k=0$. 
By equation \eqref{equ4} we obtain $n_{3,14}+n_{6,11}=0$ so they are both equal to 0. We get a contradiction
to equation \eqref{equ5}. Finally if $g(C)>1$ we have $(m_{\sigma^4},r_{\sigma^4},l_{\sigma^4})=(4,6,8)$ by \cite[Theorem 4.1]{artsa}. 
So by the same argument as in Proposition \ref{genus} this case is not possible since $A=2$.
\end{proof}
\begin{pro}\label{lefschetz8}
Let $\sigma$ be a purely non-symplectic automorphism of order 16 on a $K3$ surface $X$ and $C\subset \rm{Fix(\sigma^2)}$. 
Then $g(C)\leq 1$ and  the following relations for the number of fixed points and curves by $\sigma^2$ hold:
\begin{eqnarray*}
\begin{array}{lll}
n_{2,7} + n_{3,6}& =& 2 + 4 k_{\sigma^2},\\
n_{4,5} + n_{2,7} - n_{3,6}&= &2 + 2 k_{\sigma^2},\\
N_{\sigma^2} &= &2 +  r_{\sigma^2} - l_{\sigma^2} - 2k_{\sigma^2}.
\end{array}
\end{eqnarray*}
where $n_{i,j}$ denote the number of fixed points of type $(i,j)$ for the action of $\sigma^2$. 
\end{pro}
\begin{proof}
Observe that by Proposition \ref{fix4} we have $g(C)\leq 1$ moreover an isolated fixed point for $\sigma^2$ is given by the local action 
$ \left(
  \begin{array}{ c c }
     \xi^i & 0 \\
    0 & \xi^{j}
 \end{array} \right)$, $i+j=1 \mod (8)$. Thus by the holomorphic and topological Lefschetz formulas we have the relations in the statement.
\end{proof}
\begin{remark}\label{suite8}
By Lemma \ref{tree} and with the same notation there the local action of $\sigma^2$ at the intersection points of the curves  $R_i$ appear in the following order:
$$
\ldots, (0,1), (7,2), (6,3), (5,4), (4,5), (3,6), (2,7), (1,0),\ldots
$$
moreover the $\sigma$-fixed points of type $(5,12)$ and $(4,13)$ give $\sigma^2$ fixed points of type $(4,5)$,
the $\sigma$-fixed points of type $(2,15)$ and $(7,10)$ give $\sigma^2$ fixed points of type $(2,7)$ (up to the order). The $\sigma$-fixed points of type $(3,14)$ and $(6,11)$ give $\sigma^2$ fixed points of type $(3,6)$ (up to the order).
\end{remark}
\section{Elliptic Fibrations}
\begin{theorem}\label{elliptic}
Let $C\subset \Fix(\sigma^4)$, if $g(C)=1$ then $\sigma$ acts as an automorphism of order four on $C$  and  we have 
the following cases
\begin{center}
 \begin{tabular*}{0.90\textwidth}{@{\extracolsep{\fill}} c c c c c |c     c|c}
   
  $m^2_{\sigma}$& $m^1_{\sigma}$ & $m_{\sigma}$ & $l_{\sigma}$ & $r_{\sigma}$& $N_{\sigma}$ & $k_{\sigma}$& type of $C^{'}$\\
   \hline 

1&1&0&1&9& 8&1& $\rm{IV}^\ast$\\
\hline
    1&1&0&3&7&6&0 &$\rm{IV}^\ast$\\
         \end{tabular*}
\end{center}
where $C'$ denotes the invariant reducible fiber in the fibration determined by $C$. In particular in this case $\rk\Pic(X)=14$. 
\end{theorem}
\begin{proof}
If $g(C)=1$ we are in the case $(m_{\sigma^4}, r_{\sigma^4}, l_{\sigma^4})=(4,10,4)$ by \cite[Theorem 3.1]{artsa} and equations \eqref{ranks}. 
Moreover the curve $C$ must be $\sigma$-invariant and so the elliptic fibration induced by $C$ is preserved. 
By \cite[Theorem 3.1]{artsa} $\sigma$ has order 16 on the basis of the fibration, leaving invariant the fiber 
$C$ and the singular fiber fiber $C':=IV^*$. The latter corresponds to the other fixed point for the action of $\sigma$ on the basis $\IP^1$.
By Proposition \ref{ratio} the curve $C$ can not be fixed by $\sigma$, hence $\sigma$ has order $2$ or $4$ on it or it is a translation. There are two possible actions on $C'$. 

{\bf First case}: $IV^*$ contains a fixed rational curve, which is necessarily the component
of multiplicity 3. Then by using the Lemma \ref{tree} and the formulas in Proposition \ref{lefschetz16} we find $N=8$ with,
$k=1$, $n_{2,15}=n_{3,14}=3$ and $n_{4,13}=2$ the others $n_{i,j}$ are zero. In particular $\sigma$ must have two fixed points on $C$ this means that it acts as an automorphism of order four.  

{\bf Second case}: $IV^*$ has a simmetry of order 2. Then the curve of multiplicity 3 contains two isolated fixed
points whith action $(8,9)$. Combining Remark \ref{nn} and Proposition \ref{lefschetz16} we find  $(N,k)=(6,0)$, with $n_{8,9}=2=n_{5,12}$, $n_{7,10}=1=n_{6,11}$, the other $n_{i,j}$ are zero. We observe that also in this case $\sigma$ must have two fixed points on $C$, this means that it acts as an automorphism of order four.

Using the fact that $(m_{\sigma^4}, r_{\sigma^4}, l_{\sigma^4})=(4,10,4)$ we get immediately that in both cases $m^2_{\sigma}=m^1_{\sigma}=1$. Moreover we have that 
$r_{\sigma}+l_{\sigma}+2m_{\sigma}=10$ and in the first case we have $r-l=8$  and in the second case $r-l=4$. 
In both cases we have $N_{\sigma^2}=10$ and $k_{\sigma^2}=1$ so using Proposition \ref{lefschetz8} we obtain the
values of $r,l,m$ given in the table.


\end{proof}
\begin{example}\label{fibr0}
 {\rm Consider the elliptic fibration in Weierstrass form given by :
$$y^2=x^3+ax+bt^8$$
where  $\sigma(x,y,t)=(-x,iy,\zeta^{13}_{16}t)$. By  making the coordinate transformation that replace $x$ by $\lambda^4 x$ and $y$ by $\lambda^6 y$ for a suitable $\lambda\in \IC$ we can assume that $a=1$. Moreover since $b\not=0$ we can apply a coordinate tranformation to $t$ and assume that $b=1$ too. So our equation becomes:
$$y^2=x^3+x+t^8.$$
The fibers preserved by $\sigma$ are over $0,\infty$ and the action at infinity is (see \cite[\S 3]{kondoell}):
$$(x/t^4,y/t^6,1/t)\longmapsto (-ix/t^4,\zeta^{6}_{16}y/t^6,\zeta^{15}_{16}1/t).$$
The discriminant of the fibration is 
$$\Delta(t)=4+27t^{16}.$$
We have that $t=\infty$ is an order eight zero of $\Delta(t)=0$, and $\Delta(t)$ has $16$
simple zeros. Looking in the classification of singular fibers of elliptic fibrations
on surfaces (e.g. \cite[Section 3]{miranda}) we see that the fiber over $t=\infty$ is of type  
$\rm{IV}^\ast$ and the fibration has $16$ fibers of type $I_1$. 
In particular the fiber over $t=0$ is smooth.
By \cite[\S 3]{kondoell} a holomorphic two form is given by $(dt\wedge dx)/2y$ and so the action of $\sigma$ on it is by multiplication by $\zeta_{16}$. In fact we can be even more precise to understand the local action of the automorphism $\sigma$
at the fixed points on $C$. 
If we look at the elliptic fibration locally around the fiber over $t=0$ the equation in $\IP^2\times \IC$ is given by:
$$G(x,y,z,t):=zy^2-(x^3+z^2x+z^3t^8)=0$$
where $(x:y:z)$ are the homogeneous coordinates of $\IP^2$ and the two fixed points for the automorphism $\sigma$ on the fiber $t=0$ are $p_0:=(0:1:0)$ and $p_1:=(0:0:1)$. In the chart $z=1$ and on the open subset $\partial G(x,y,1,0)/\partial x\not=0$ that contains the fixed point $p_1=(0:0:1)$ a one form for the elliptic curve over $t=0$ is:
$$dy/(\partial G(x,y,1,0)/\partial x)=dy/(-3x^2-1)$$
Here the action of $\sigma$ is a multiplication by $i$ so that the action on the holomorphic two form:
 $$dt\wedge (dy/(-3x^2-1))$$
  is a multiplication by  $\zeta_{16}$ as expected, and we see that the local action is of type $(4,13)$. 
Doing a similar computation in an open subset of the chart $y=1$ that contains the fixed point $p_0$ we find again
the same local action. So we are  in the first case of the Proposition \ref{elliptic} with $N=8$. On the other hand the fibration admits also the automorphism $\gamma(x,y,t)=(-x,-iy,\zeta^{5}_{16}t)$. This acts also by multiplication by $\zeta_{16}$ on the holomorphic two form, so $\gamma$ is not a power of $\sigma$. In this case a similar computation as above shows that the local action at the fixed points on the fiber $C$ is of type $(5,12)$, so we are in the second case of the Proposition \ref{elliptic}.}
\end{example}
\begin{pro}\label{invfibr}
 Let $\sigma$ be a purely non-symplectic automorphism of order 16 on a $K3$ surface $X$ such that $\Pic (X)=S(\sigma^8)\cong U\oplus L$ where 
 $L$ is isomorphic to a direct sum of root lattices of types $A_1 , D_{4+n} , E_7 $ or $E_8$ and $\sigma^8$ fixes a curve of genus $g>1$ . Then $X$ 
 carries a jacobian elliptic fibration $\pi:X\longrightarrow \mathbb{P}^1$ whose fibers are $\sigma^8-$invariant and it has reducible fibers 
 described by $L$ and a unique section $E\subset \rm{Fix}(\sigma^8)$. Moreover , if $g>4$ then $\pi$ is $\sigma-$invariant .
\end{pro}
\begin{proof}
 Since $\rm{Pic}(X)=S(\sigma^8)\cong U\oplus L$ the first half of the statement follows from \cite[Lemma 2.1, 2.2]{kondoell}. On other hand, 
 since $\sigma^8$ fixes a curve $C$ of genus $g>1$, then $C$ intersects each fiber of $\pi$ in at least two points. This implies that $\sigma^8$
 preserves each generic fiber of $\pi$ and acts on it as an involution with four fixed 
 points. By \cite[Theorem 6.3]{SS} we have that the Mordell-Weill group of $\pi$ is $MW(\pi) \cong \Pic (X)/T$ where $T$ denote 
 the subgroup of $\Pic(X)$ generated by the zero section and fiber components. Since $L$ is a root lattice and $\Pic (X)\cong U \oplus L$ 
  we have that $MW(\pi)$ is trivial, hence $\pi$ has a unique section $E$. Since $\sigma^8$ preserves each fiber of $\pi$ and $E$ is invariant, we have that $E$ is fixed by $\sigma^8$. 
 This implies that $C$ intersects each fiber in three points and one fixed point for the action of $\sigma^8$ is contained in the  section $E$.
 
  Now we will prove that $\pi$ is $\sigma-$invariant if $g>4$ . Let $f$ be the class of a fiber of $\pi$.
The automorphisms $\sigma$ preserves the curve $C$, and we have that $CE=0$ (the fixed curves for $\sigma^8$ can not intersect). Assume
that $f\not=\sigma^*(f)$ then they intersect in at least  $2$ points. In fact if $f\sigma^*(f)=1$ then this is a fixed point on $f$ and so 
either $C$ is fixed by $\sigma$ which is not possible, or $E$ is fixed by $\sigma$. This is not possible too, since otherwise
each fiber would admit an automorphism of order $16$. Hence $\sigma$ is a translation, which is impossible too. 
Now applying \cite[Lemma 5]{artsa} we find that:
$$
2g-2=C^2\leq \frac{2(C\cdot f)^2}{f\cdot \sigma^*(f)+1}\leq \frac{2\cdot 9}{3}=6
$$ 
This implies $g=g(C)\leq 4$. 
\end{proof}
\section{The rank six case}

\begin{theorem}\label{rank6}
 Let $\sigma$ be an autormorphism of order $16$ acting purely non-symplectically on a $K3$ surface $X$ and assume that $\rm{Pic}(X)=S(\sigma^8)$ has rank $6$. Then $\sigma$ fixes at most one rational curve. 

The corresponding invariants of $\sigma$ are given in the Table below. In any case $n_{4,13}=n_{5,12}=n_{6,11}=0$ and we have $(n_{2,15}, n_{3,14}, n_{7,10}, n_{8,9})=(4,1,1,0)$ in one case and  $(n_{2,15}, n_{3,14}, n_{7,10}, n_{8,9})=(0,1,1,2)$ in the other case.

\begin{center}
 \begin{tabular*}{0.90\textwidth}{@{\extracolsep{\fill}} c c c c c |c  c c c  |c}
   
  $m_\sigma^2$& $m_\sigma^1$ & $m_{\sigma}$ & $l_{\sigma}$ & $r_{\sigma}$& $N_{\sigma}$ & $k_{\sigma}$& $N'$ &$g(C)$& $\rm{Pic}(X)$\\
   \hline 
2&0&0&0&6&6&1&4 &7 & $U\oplus D_4$ \\
\hline
2&0&0&2&4&4&0&2&6& $U(2)\oplus D_4$\\

         \end{tabular*}\\
\end{center}
Here $C$ denotes the $\sigma^8$-fixed curve  of genus $>1$ and $N'$ denotes the number of fixed points that are contained in $C$. 
\end{theorem}
\begin{proof}

 By the classification theorem for non-symplectic involutions on $K3$ surfaces given by Nikulin in \cite[\S 4]{nikulinfactor} we have that $(g(C), k_{\sigma^8})$ is either equal to $(5,0) , (6,1)$ or $(7,2)$ .  Observe that the 
case $g(C)=5$ is not possible. In fact in this case since $k_{\sigma^8}=0$ then $k_{\sigma^4}=0$ too and since $C$ is not fixed by $\sigma^4$ by Proposition \ref{fix4}, we get a contradiction awith Proposition \ref{fix4} again. Observe that we have $m^2_{\sigma}=2$ so that $m_{\sigma^4}=8$ by formulas \eqref{ranks}. This means that the automorphism $\sigma^4$ can not have $l_{\sigma^4}\geq 0$ by \cite[Theorem 8.1]{artsa}. This implies that $l_{\sigma^4}=0$ and by \cite[Theorem 6.1]{artsa} or \cite[Main Theorem 1]{Taki} we have two possible cases that we recall below, both have $m^1_{\sigma}=0$.

\underline{The case $(g(C),k_{\sigma^8})=(6,1)$}. The automorphism $\sigma^4$ of order $4$ fixes one rational curve and six points on $C$ by \cite{artsa}, \cite{Taki}. 
By Riemann-Hurwitz formula applied to the automorphism 
$\sigma$ on $C$ we find that either $\sigma$ exchanges two fixed points and permutes the other four or $\sigma$   
fixes two points and the other four are exchanged two by two.
The first case is not possible since then $N=2$ and by Proposition \ref{lefschetz16} we know that $N\geq 4$. 
So we are in the second case. Since again $N\geq 4$ then the rational
curve is invariant but not fixed and so $N=4$ and by Remark \ref{nn} we have $(n_{3,14},n_{7,10},n_{8,9})=(1,1,2)$ the others $n_{i,j}$ are zero. We have moreover that $k_{\sigma^2}=1$ and $N_{\sigma^2}=6$ so
combining the Lefschetz formulas we have $r+l+2m=6$, $4=2+r-l$, $6=2+r+l-2m-2$. That gives $m=0$ and  $r=4$, $l=2$. This is the second case in the table.

\underline{The case $(g(C),k_{\sigma^8})=(7,2)$}. The automorphism $\sigma^4$ of order $4$ fixes one rational curve, four points on $C$ and 
two points on the other rational curve see 
\cite{artsa}, \cite{Taki}. By Riemann-Hurwitz formula applied to the automorphism 
$\sigma$ on $C$ we find that either $\sigma$ exchanges two by two the four points or it fixes each of the four points. In the first case since $N\geq 4$ we have that the two 
rational curves are invariant and they contain 2 fixed points each, so that $N=4$ by Remark \ref{nn}. Then $(n_{3,14},n_{7,10},n_{8,9})=(1,1,2)$ so that $k_{\sigma^2}=1$ and $N_{\sigma^2}=6$. We have  $n_{2,7}+n_{3,6}=6$ and 
since $n_{4,5}=0$ (we have $k_{\sigma^4}=1$) we get $n_{3,6}=1$ and $n_{2,7}=5$.  Using Proposition \ref{lefschetz16} and \ref{lefschetz8} we compute here that $(r,l,m)=(4,2,2)$ and we have also $\Pic(X)=U\oplus D_4$ by \cite[Theorem 6.1]{artsa}. 
By applying Proposition \ref{invfibr} we know that the K3 surface $X$ carries a $\sigma$-invariant elliptic fibration 
with a singular fiber $I_0^*$. Since the action is not trivial on $\Pic(X)$ the automorphism 
$\sigma$ should act non trivially on $I_0^*$. Since $C$ intersects in three points the fiber $I_0^*$ then the only possibility is that $\sigma$ exchanges two components of multiplicity one. Then the
third point on $C$ would be fixed but this is not possible. So the action of $\sigma$ on $C$ fixes the four points. Observe that then the number of fixed points for $\sigma^2$ 
satisfies $n_{2,7}+n_{3,6}\geq 4$ so that $k_{\sigma^2}=1$ by Proposition \ref{lefschetz8}. This again gives $n_{2,7}+n_{3,6}=6$  and so $n_{4,5}=0$ and $n_{2,7}=5$,$n_{3,6}=1$.  Finally Observe that the case $(N,k)=(8,0)$ is not possible for $\sigma$ by Remark \ref{nn} and so we have $(N,k)=(6,1)$. Again by Remark \ref{nn} we have $(n_{2,15},n_{3,14},n_{7,10})=(4,1,1)$. In this case we have $r+l+2m=6$, $r-l=6$, $r+l-2m=6$. We find $m=0$,
$r=6$, $l=0$. So $\sigma$ acts trivially on $\Pic(X)$ and this is the first case in the table.

 \end{proof}
\begin{example}\label{fibr1}
{\rm 1) The case $g(C)=7$, $(r_{\sigma},l_{\sigma})=(6,0)$, $\Pic(X)=U\oplus D_4$. 

Consider as in \cite[Section 3.4]{matthias} the elliptic fibration:
$$
y^2=x^3+t^2 x+(bt^3+t^{11})
$$
with automorphism $\sigma(x,y,t)=(\zeta_{16}^2 x, \zeta_{16}^3 y, \zeta_{16}^2 t)$ (we write here the fibration in a slightly
different way as given in \cite{matthias}).
On $t=0$ the fibration has a fiber $I_0^*$ and on $t=\infty$ the fibration has a fiber $II$. 
The action on the holomorphic two form $(dx\wedge dt)/2y$ is a multiplication by $\zeta_{16}$. This is a one dimensional family and for generic $\lambda$ the action is trivial
on $\Pic(X)$. So we are in the second case of Theorem \ref{rank6}. Observe that the fiber $I_0^*$ contains the four fixed points with local action of type $(2,15)$ and the invariant elliptic cuspidal curve over $t=\infty$ contains the fixed point with local action $(14,3)$ (which is also contained
on the section of the fibration) and the point of type $(7,10)$. In particular observe that the curve $C$ of genus 7 meets with multiplicity 3
the  fiber $II$ at the singular point.\\
Observe that if $b=0$ we get the elliptic fibration with the order 32 automorphism
$$\sigma_{32}(x,y,t)=(\zeta_{32}^{18} x, \zeta_{32}^{11} y, \zeta_{32}^2 t)$$
as described e.g. in \cite{Taki32}. The automorphism $\sigma$ is the square
of the automorphism $\sigma_{32}^{25}$. 

{\rm 2)} The case $g(C)=6$, $(r_{\sigma},l_{\sigma})=(4,2)$, $\Pic(X)=U(2)\oplus D_4$. 

The surfaces of this kind are described in the paper \cite{laza} and
they are double covers of $\IP^2$ ramified on a reducible
sextic which is the product of a smooth quintic and a line.
We consider the special family with equation
in $\IP(3,1,1,1)$:
$$
z^2=x_0(\alpha_0 x_0^4x_2+\beta_0 x_1^5+\beta_1 x_1^3x_2^2+\beta_2 x_1x_2^4).
$$
Observe that the quintic curve is smooth and the K3 surface has five
$A_1$ singularities over the points of intersection
of the quintic curve and the line.
The K3 surface carries the order $16$ non-symplectic automorphism
$$
\sigma(z:x_0:x_1:x_2)\mapsto(\zeta_{16}^3 z: x_0: \zeta_8^7 x_1:\zeta_8^3 x_2).
$$
This acts by multiplication by $\zeta_{16}$ on the 
holomorphic two form:
$$
(dx\wedge dy)/ \sqrt{f}
$$
where $f(x,y)=0$ is the equation of the ramification sextic 
in the local coordinates $x$ and $y$. 
An easy computation shows that the automorphism fixes the points:
$$
(0:1:0:0), \qquad (0:0:1:0), \qquad (0:0:0:1)
$$
Observe that the point $(0:0:0:1)$ is in fact one of the five $A_1$ singularities
on the K3 surface. If we resolve it we find a fixed point on the strict transform of $C$ which is
  the quintic curve on $\IP^2$ (that have genus six) and one fixed point on the strict transform
 of $L$ which denotes the curve $\{x_0=0\}$. The other two fixed points are contained respectvely in $C$ and $L$ (and their respectives 
strict
transforms).
 Observe that the automorphism $\sigma$ exchanges two by two the other points of intersection of $C$ with $L$.
}
\end{example}

\section{The rank fourteen case}

\begin{theorem}\label{rank14}
Let $\sigma$ be an automorphism of order $16$ acting purely non symplectically 
on a K3 surface $X$ and assume that $S(\sigma^8)=\Pic(X)$ has rank $14$. Then the 
surface K3 is one of the surfaces described in Proposition \ref{elliptic} with a fixed elliptic curve 
for the automorphism $\sigma^4$ or it has:
\begin{center}
\begin{tabular*}{0.90\textwidth}{@{\extracolsep{\fill}} c c c c c |c  c c c  |c}
   
  $m_\sigma^2$& $m_\sigma^1$ & $m_{\sigma}$ & $l_{\sigma}$ & $r_{\sigma}$& $N_{\sigma}$ & $k_{\sigma}$& $N'$ &$g(C)$& $\rm{Pic}(X)$\\
   \hline 
1&0&0&1&13&10&1& 2&3 & $U\oplus D_4\oplus E_8$\\
\hline
1&0&1&1&11&8&1&2&2 &$U(2)\oplus D_4\oplus E_8$\\
\hline
1&0&1&5&7&2&0&2&2 &$U(2)\oplus D_4\oplus E_8$\\
\end{tabular*}\\
\end{center}
Here $C$ denotes the $\sigma^8$-fixed curve of genus $>1$ and $N'$ denotes the number of fixed points that are contained in $C$. 
\end{theorem} 
\begin{proof} 
By \cite[\S 4]{nikulinfactor} we know that for the genus $g:=g(C)$ of the fixed curve 
by $\sigma^8$ and the number $k_{\sigma^8}$ of rational curves (different from $C$) holds:
$$
(g, k_{\sigma^8})=(0,3),~(1,4),~(2,5),~(3,6)
$$

{\bf The case $g(C)=0$}. We are in the case of \cite[Theorem 5.1]{artsa}
for $\sigma^4$, so we have $(r_{\sigma^4}, l_{\sigma^4}, m_{\sigma^4})=(10,4,4)$
and since $N_{\sigma^4}=6$ and $k_{\sigma^8}=3$, we have $N_{\sigma}=4, 6, 8$ by Proposition \ref{lefschetz16}.
Moreover since $k_{\sigma^4}=1$ then $k_{\sigma^2}$ and $k_{\sigma}$ are $0$ or $1$.

Assume first $k_{\sigma^2}=0$ since $\sigma^4$ acts in a different way on the four rational curves, these must be preserved by $\sigma$ and so also $\sigma^2$. We have $n_{4,5}=2$, $n_{2,7}=3=n_{3,6}$ by Remark \ref{suite8}. These contradicts
Proposition \ref{lefschetz8}. 
If $k_{\sigma^2}=1$ then $n_{4,5}=0$ and $n_{2,7}=3=n_{3,6}$. This again contradicts Proposition \ref{lefschetz8}.
 
{\bf The case $g(C)=1$}. We can assume $C\not\subset \Fix(\sigma^4)$ otherwise we have discussed this case already in 
Theorem \ref{elliptic}. Since $C$ is fixed by $\sigma^8$ then $C$ is also $\sigma$-invariant. Hence $\sigma$ acts as a translation
on $C$ (otherwise $C$ would admits an automorphism of $2$-power order bigger than 4, which is not possible). So that 
$C$ does not contain fixed points for $\sigma$. By \cite[Theorem 8.4]{artsa} we get that $N_{\sigma}=4, 6, 8$ and $k_{\sigma}=1$ or $0$.
Studying the action of $\sigma^2$ on the four rational curves and using the same argument as before, one shows easily that this
case is not possible.

{\bf The case $g(C)=2$}. By Proposition \ref{fix4} we have $k_{\sigma^4}\geq 1$ so that $\sigma^4$ fixes at least a rational curve.  Moreover by formulas \eqref{ranks} we have $r_{\sigma^4}+l_{\sigma^4}=14$ and $l_{\sigma^4}, m_{\sigma^4}\in 4\IZ$.  Observe that $m_{\sigma^4}=4m^2_{\sigma}=4$. By \cite[Theorem 8.1]{artsa} if $l_{\sigma^4}>0$ then we have $l_{\sigma^4}+m_{\sigma^4}=4$ or $8$. The first case is not possible, if $l_{\sigma^4}+m_{\sigma^4}=8$ then $l_{\sigma^4}=4$ and by \cite[Theorem 8.1]{artsa} we have $k_{\sigma^4}=1$. Observe that $\sigma$ preserves or permutes two by two the four rational curves 
not fixed by $\sigma^4$ so that in any case $N_{\sigma^4}\geq 8$. By \cite[Proposition 1]{artsa} 
we have $N_{\sigma^4}=6$ which contradicts the previous inequality.
Hence $l_{\sigma^4}=0$ and so $\sigma^4$ acts trivially on $\Pic(X)$. By \cite[Theorem 6.1]{artsa} we have
$(m_{\sigma^4}, r_{\sigma^4}, n_1,n_2,k_{\sigma^4})=(4,14,4,6,3)$ where $N_{\sigma^4}=n_1+n_2$ and $n_2$ is the number of fixed points on $C$. So we have $4$ points contained in the two rational curves that are $\sigma^4$-invariant but not fixed. We call these curves $R_1$ and $R_2$. We study now the action of $\sigma$ and $\sigma^2$ on the $5$ rational curves, fixed by $\sigma^8$, and on $C$.

\underline{The automorphism $\sigma^2$}. We have $k_{\sigma^2}\leq 3$ and at least one of the five curves is preserved or fixed. By using Remark \ref{suite8} we have: $n_{4,5}\in 2\IZ$ (points of this type can occur only on the rational curves) and $n_{2,7}+n_{3,6}\leq 10$ (we have $N_{\sigma^4}=10$,  at most 6 fixed points are on $C$ and points of this type are not contained on rational curves that are fixed for $\sigma^4$ but can be contained
in the two rational curves that are only $\sigma^4$-invariant). By using Proposition \ref{lefschetz8} we obtain that $k_{\sigma^2}\leq 2$. If $k_{\sigma^2}=0$, since the action of $\sigma^4$ is not the same, then all the rational curves are preserved by $\sigma^2$ in particular $n_{4,5}=6$ and $n_{2,7}\geq 2$ $n_{3,6}\geq 2$. This contradicts Proposition \ref{lefschetz8}. We are left with the cases with $k_{\sigma^2}=1$ or $k_{\sigma^2}=2$.
 
i) $k_{\sigma^2}=2$. By Proposition \ref{lefschetz8} we get $n_{2,7}+n_{3,6}=10$ this means that the curve $C$ must contain six fixed points for $\sigma^2$ and the other four fixed points are contained in the two $\sigma^4$-invariant curves $R_1$ and $R_2$. In particular we have $n_{2,7}\geq 2$ and $n_{3,6}\geq 2$, and $n_{4,5}=2$. Since by Proposition \ref{lefschetz8} we have 
 $n_{4,5}=2 n_{3,6}-4$ we get $n_{3,6}=3$, $n_{2,7}=7$, $N_{\sigma^2}=12$. 

ii) $k_{\sigma^2}=1$. By Proposition \ref{lefschetz8} we have $n_{2,7}+n_{3,6}=6$. Observe that for the same reason as above the remaining rational curves can not be exchanged two by two. So these are invariant. This gives  $n_{2,7}\geq 2$, $n_{3,6}\geq 2$ and $n_{4,5}=4$. Using Proposition \ref{lefschetz8} we obtain that $n_{2,7}=n_{3,6}=3$. And two fixed points are contained in $C$. The other points on $C$ fixed by $\sigma^4$ form a $\sigma$-orbit of length four.

\underline{The automorphism $\sigma$}. First observe that using Riemann-Hurwitz formula on $C$ we have two possibilities: 
$C$ contains 2 fixed points and the other four points are permuted by $\sigma$ in one orbit (this is case ii)) or the six points are exchanged two by two and so fixed by $\sigma^2$ (this is case i)).

i) In this case $\sigma$ exchanges two by two the points on $C$. We have $n_{5,12}=n_{4,13}=1$
 since these two points correspond to the two fixed points with local action $(4,5)$ for $\sigma^2$ and are contained on a rational curve (see Remark \ref{suite}). Assume that $R_1$ and $R_2$ are not exchanged. We have $n_{2,15}+n_{7,10}+n_{3,14}+n_{6,11}=4$ and $n_{2,15}=n_{3,14}$, $n_{7,10}=n_{6,11}$. But this contradicts equation \eqref{equ5} in Proposition \ref{lefschetz16}. 
If $R_1$ and $R_2$ are exchanged we have $n_{3,14}=n_{6,11}=0$, $n_{2,15}=n_{7,10}=0$ and $n_{5,12}=n_{4,13}=1$. But this contradicts the equality $n_{3,14}-n_{6,11}=1$ in Proposition \ref{lefschetz16}. 

ii) In this case $C$ contains two fixed points for $\sigma$. We have $n_{8,9}=2w$, with $w=0,1$. Moreover by Remark \ref{suite} we have $n_{5,12}=n_{4,13}=2$ or $n_{5,12}=n_{4,13}=0$. If $n_{8,9}=2$ so that $k_{\sigma}=0$ an easy computation using the equations
of Proposition \ref{lefschetz16} shows that the first case with $n_{5,12}=n_{4,13}=2$ is not possible. If $n_{5,12}=n_{4,13}=0$ again using Proposition \ref{lefschetz16} 
we find that $n_{3,14}=n_{7,10}=1$ the other $n_{ij}$ are zero. One computes $(r_{\sigma},l_{\sigma},m_{\sigma})=(7,5, 1)$ and we have 
  $\Pic(X)=U(2)\oplus D_4\oplus E_8$. Observe that in this case the remaining $\sigma^8$-fixed rational curves are exchanged two by two
by $\sigma$. If $n_{8,9}=0$ so that $k_{\sigma}=1$ again one computes using Proposition \ref{lefschetz16} that :
$$
(N, k, n_{8,9}, n_{2,15},n_{3,14}, n_{4,13}, n_{5,12},n_{6,11}, n_{7,10})=(10, 1, 0, 3,2,2,2,1,0)
$$
and $(r_{\sigma},l_{\sigma},m_{\sigma})=(11,1, 1)$. Moreover we have $\Pic(X)=U(2)\oplus D_4\oplus E_8$, .

{\bf The case $g(C)=3$}. By Proposition \ref{fix4} we have $k_{\sigma^4}\geq 1$ so that $\sigma^4$ fixes at least a rational curve.  We have moreover by formulas \eqref{ranks} that $r_{\sigma^4}+l_{\sigma^4}=14$ and $l_{\sigma^4}, m_{\sigma^4}\in 4\IZ$ and observe that $m_{\sigma^4}=4m^2_{\sigma}=4$. By \cite[Theorem 8.1]{artsa} if $l_{\sigma^4}>0$ then we have $l_{\sigma^4}+m_{\sigma^4}=4$ or $8$. The first case is not possible, if $l_{\sigma^4}+m_{\sigma^4}=8$ then $l_{\sigma^4}=4$ and by \cite[Theorem 8.1]{artsa} we have $k_{\sigma^4}=1$. Observe that $\sigma$ preserves or permutes some of the five rational curve not fixed by $\sigma^4$ so that in any case $N_{\sigma^4}\geq 10$. By \cite[Proposition 1]{artsa} 
we have $N_{\sigma^4}=6$, which is not possible. Hence $l_{\sigma^4}=0$ and so $\sigma^4$ acts trivially on $\Pic(X)$. By \cite[Theorem 6.1]{artsa} we have
$(m_{\sigma^4}, r_{\sigma^4}, n_1,n_2,k_{\sigma^4})=(4,14,6,4,3)$ where $N_{\sigma^4}=n_1+n_2$ and $n_2$ is the number of fixed points on $C$. We have hence $6$ points contained in the three rational curves that are $\sigma^4$-invariant but not fixed. We call these curves $T_i$, $i=1,2,3$. We study now the action of $\sigma$ and $\sigma^2$ on the $6$ rational curves fixed by $\sigma^8$ and on $C$.

\underline{The automorphism $\sigma^2$}. We have $k_{\sigma^2}\leq 3$ and observe that since $\sigma$ can not permute the four curves, since the action of $\sigma^4$ is different, then each curve is preserved by $\sigma^2$.  Moreover we  have $n_{4,5}\in 2\IZ$, and these are at most $6$, in fact points of this type can occur only on the rational curves,   and $n_{2,7}+n_{3,6}\leq 10$ (we have at most 4 fixed points on $C$ and points of this type are not contained on rational curves that are fixed for $\sigma^4$, but can be contained
in the three rational curves that are only $\sigma^4$-invariant). Again by using Proposition \ref{lefschetz8} we find  that $k_{\sigma^2}\leq 2$. If $k_{\sigma^2}=0$ then $n_{2,7}+n_{3,6}=2$ but since all the rational curves are preserved $n_{4,5}=6$ and we get a contradiction using Proposition \ref{lefschetz8}. We are left with the cases with $k_{\sigma^2}=1$ or $k_{\sigma^2}=2$.

i) $k_{\sigma^2}=2$. Here we get $n_{2,7}+n_{3,6}=10$ this means that the curve $C$ must contain four fixed points for $\sigma^2$ and the other six points are contained in the three $\sigma^4$-invariant curves $T_1$, $T_2$ and $T_3$. In particular we have $n_{2,7}\geq 3$ and $n_{3,6}\geq 3$, $n_{4,5}=2$. Moreover $n_{4,5}=2 n_{3,6}-4$ so we get  $n_{3,6}=3$, $n_{2,7}=7$, $N_{\sigma^2}=12$ (by Proposition \ref{lefschetz8}). 

ii) $k_{\sigma^2}=1$: Here we get $n_{2,7}+n_{3,6}=6$ by Proposition \ref{lefschetz8}. Observe that for the same reason as above the remaining rational curves can not be exchanged two by two. So these are invariant. This gives  $n_{2,7}\geq 3$, $n_{3,6}\geq 3$ and $n_{4,5}=4$. We get using Proposition \ref{lefschetz8} that $n_{2,7}=n_{3,6}=3$, and so the four points on $C$ fixed by $\sigma^4$ form a $\sigma$-orbit of length four.

\underline{The automorphism $\sigma$}. By using Riemann-Hurwitz formula there are two possible actions on $C$: The automorphism $\sigma$ exchanges $2$ points and fixes the other two (this is case i)) or the four points form a $\sigma$-orbit (this is case ii)).

i) We have $n_{8,9}=2w$ and since $k_{\sigma^2}=2$ we have $0\leq w\leq 2$.  Moreover $n_{5,12}=n_{4,13}=1$
(since these two points correspond to the two fixed points with local action $(4,5)$ for $\sigma^2$). If $w=0$ and $k=0$, so that the two $\sigma^2$-fixed curves are exchanged by $\sigma$, then using Proposition \ref{lefschetz16} one sees that this case is not possible. If $w=0$ and   $k=2$ using Proposition \ref{lefschetz16} we get $N=14$ which is impossible by looking at the geometry (in fact in this case we have $N\leq 12$). 

If $w=1$, then $k=1$ and we find $N=12$ with
$$
(N, k, n_{8,9}, n_{2,15},n_{3,14}, n_{4,13}, n_{5,12},n_{6,11}, n_{7,10})=(12, 1, 2, 3,2,1,1,1,2)
$$
This is the case in the statement. 

If $w=2$ and $k=0$ this is not possible by using equation in Proposition \ref{lefschetz16}. 
 
ii) We have $n_{8,9}=2w$ and since $k_{\sigma^2}=1$ we have $w=0, 1$. If $w=0$ then $k=1$ and $n_{5,12}=n_{4,13}=2$ or $n_{5,12}=n_{4,13}=0$. If $n_{5,12}=n_{4,13}=2$ we obtain 
$n_{6,11}=1$ and $n_{7,10}=0$ which is impossible since the fixed points by $\sigma$ are contained in the rational curves that are fixed by $\sigma^8$ (see Remark \ref{suite}). If $n_{5,12}=n_{4,13}=0$ then two of the $\sigma^4$-fixed curves are exchanged. By using Proposition \ref{lefschetz16} we get $n_{7,10}=1$, $n_{2,15}=4$, $n_{3,14}=1$ (the other $n_{ij}$ are zero), but this is not possible since the isolated points fixed by $\sigma$ are contained in rational curves (see Remark \ref{suite}). 
 
If $w=1$ then $k=0$ then again $n_{5,12}=n_{4,13}=2$ or $n_{5,12}=n_{4,13}=0$. By using Proposition \ref{lefschetz16} we see that the first case is not possible. If $n_{5,12}=n_{4,13}=0$ then two of the $\sigma^4$-fixed curves are exchanged. By Proposition \ref{lefschetz16} we find $N=4$. This is not possible in fact if the curves $T_i$ are preserved then $N=6$, if two of them are exchanged we get $N=2$. In any case we get a contradiction.  

\end{proof}
\begin{example}\label{fibr2}
{1) \bf The case $g(C)=3$} (see \cite{Taki}). Consider the elliptic  fibration: 
$$
y^2=x^3+t^2x+t^7
$$
This carries the order 16 automorphism $\sigma(x,y,t)=(\zeta_{16}^{2} x, \zeta_{16}^{11} y, \zeta_{16}^{10} t)$.
The discriminant is $t^6(4+27 t^{8})$ so over $t=0$ the fibration has a fiber $I_0^*$ and over $t=\infty$
the fibration has a fiber $II^*$. The automorphism $\sigma$ preserves the $II^*$ fiber and fixes the component 
of multiplicity 6. The genus 3 curve cuts the fiber $II^*$ in the isolated component of multiplicity 3.  
Finally $\sigma$ exchanges two curves in the $I_0^*$ fiber (this corresponds to $l_{\sigma}=1$), it leaves invariant
the component of multiplicity two and contains two fixed point on it. Using Remark \ref{suite} it is easy to find the local action at the $12$ fixed points.
In this case we have $\Pic(X)=U\oplus D_4\oplus E_8$.

2) {\bf The case $g(C)=2$ and $k_{\sigma}=0$}. We consider the K3 surface double cover of $\IP^2$ ramified on a special reducible sextic as in Example \ref{fibr1}, 2). We consider the quintic with a  special equation,
more precisely we assume that the reducible sextic  $(L=\{x_0=0\})\cup C$ has the equation:
$$
x_0(x_0^4x_2+x_1^5-2x_1^3x_2^2+x_2^4x_1)=0,
$$
and recall that the automorphism is:
$$
\sigma(z:x_0:x_1:x_2)\mapsto(\zeta_{16}^3 z: x_0: \zeta_8^7 x_1:\zeta_8^3 x_2).
$$
The line  $L=\{x_0=0\}$ meets the quintic in the point $(0:0:1)$ and two further points $(0:1:1)$ and $(0:-1:1)$, that are in fact exchanged by the automorphism $\sigma$. By studying the partial derivatives of the equation
of $C$ one sees that these are singular points. These are in fact $A_3$ singularities. We explain the
computations in detail for the point $(0:1:1)$. In the chart $x_2=1$ the equation of $C$ becomes:
$$
x_0^4+x_1^5-2x_1^3+x_1=0
$$
We translate the point $(0,1)$ to the origin
and we get an equation in new local coordinates (here $x_0=y$): 

$$x^2(x^3+5x^2+8x+4)+y^4=0$$
So we have a double point at $(0,0)$ and by making a coordinates transformation as in \cite[Ch. II, section 8]{BPV}
we obtain the local equation:
$$
x^2+y^4=0
$$
which is an $A_3$ singularity. Now as explained again in \cite[Ch. II, section 8]{BPV} or also in  \cite[Lemma 3.15]{laza}
this gives a $D_6$ singularity of the reducible ramification sextic. The same happens at the point $(0:-1:1)$ since the two points are exchanged
by $\sigma$. This means that the K3 surface defined by
$$
z^2=x_0(x_0^4x_2+x_1^5-2x_1^3x_2^2+x_2^4x_1)
$$
has two $D_6$ singularities and one $A_1$ singularity (coming from the intersection point $(0:0:1)$). Let $X$ be the desingularization of the double cover. The rank of the Picard group is at least $14$ but since the automorphism of order $16$ acts non-symplectically on it, the rank is exactly $14$ and $\Pic(X)=U(2)\oplus D_4\oplus E_8$. 
Observe that the $(-2)$-curve coming from the resolution of the $A_1$ singularity can not be fixed, because it intersects $C$ and $L$ on $X$ (we call again in this way the strict transforms) that are $\sigma^8$ fixed. Moreover since the two $D_6$ singularities are exchanged
we have $k=0$. Observe that the induced automorphism on $\IP^2$ fixes also the point $(0:1:0)\in L$ and the point $(1:0:0)\in C$ which together with the two intersection points with $L$ and $C$ of the excetional $(-2)$-curve on the $A_1$ singularity  gives $N=4$. 
\begin{remark}
If $\rk S(\sigma^8)=14$ then the automorphism $\sigma$ acts on $S(\sigma^8)^{\perp}\otimes \IC$ by the eight primitive roots of unity $\zeta_{16}^i$, $i=1,3,\ldots, 15$. In particular each eigenspace is one-dimensional, so by applying the construction for the moduli space of K3 surfaces with non-symplectic automorphisms as described in \cite[Section 11]{DK}, we see that in fact this is zero dimensional. This is the case in Theorem \ref{elliptic} and in Theorem \ref{rank14}. More precisely  
in Theorem \ref{rank14} if we fix $\Pic(X)=S(\sigma^8)=U(2)\oplus D_4\oplus E_8$ and both cases exist, 
we expect to have the same K3 surface with two different automorphisms acting on it.  
If $\rk S(\sigma^8)=6$ using the same construction as above one finds that the dimension of the moduli space is one.
 \end{remark}


\end{example}

\bibliographystyle{amsplain}
\bibliography{Biblio}

\providecommand{\bysame}{\leavevmode\hbox to3em{\hrulefill}\thinspace}
\providecommand{\MR}{\relax\ifhmode\unskip\space\fi MR }
\providecommand{\MRhref}[2]{%
  \href{http://www.ams.org/mathscinet-getitem?mr=#1}{#2}
}
\providecommand{\href}[2]{#2}
\begin{thebibliography}{10}

\bibitem{dimathesis}
D.~Al~Tabbaa, \emph{{Non-symplectic automorphisms of 2-power order on K3
  surfaces}}, PhD thesis University of Poitiers, in preparation.

\bibitem{AS3}
M.~Artebani and A.~Sarti, \emph{Non-symplectic automorphisms of order 3 on
  {$K3$} surfaces}, Math. Ann. \textbf{342} (2008), no.~4, 903--921.

\bibitem{artsa}
\bysame, \emph{{Symmetries of order four on K3 surfaces}}, J. Math. Soc. Japan
  (in print).

\bibitem{ast}
M.~Artebani, A.~Sarti, and S.~Taki, \emph{{K3 surfaces with non-symplectic
  automorphisms of prime order.}}, Math. Z. \textbf{268} (2011), no.~1-2,
  507--533, with an appendix by Shigeyuki Kond{\=o}.

\bibitem{BPV}
W.~Barth, C.~Peters, and A.~Van~de Ven, \emph{Compact complex surfaces},
  Ergebnisse der Mathematik und ihrer Grenzgebiete (3) [Results in Mathematics
  and Related Areas (3)], vol.~4, Springer-Verlag, Berlin, 1984. \MR{749574
  (86c:32026)}

\bibitem{DK}
I.V. Dolgachev and S.~Kond{\=o}, \emph{Moduli of {$K3$} surfaces and complex
  ball quotients}, Arithmetic and geometry around hypergeometric functions,
  Progr. Math., vol. 260, Birkh\"auser, Basel, 2007, pp.~43--100.

\bibitem{FU}
L.~Fu, \emph{{On the action of symplectic automorphisms on the $CH_0$-groups of
  some Hyperk\"ahler Fourfolds}}, preprint, arXiv:1302.6531v1.

\bibitem{Huybrechts}
D.~Huybrechts, \emph{Symplectic automorphisms of {$K3$} surfaces of arbitrary
  order}, Math. Res. Lett. \textbf{19} (2012), 947--951.

\bibitem{kondoell}
S.~Kond{\=o}, \emph{Automorphisms of algebraic {$K3$} surfaces which act
  trivially on {P}icard groups}, J. Math. Soc. Japan \textbf{44} (1992), no.~1,
  75--98.

\bibitem{laza}
R.~Laza, \emph{Deformations of singularities and variation of {GIT} quotients},
  Trans. Amer. Math. Soc. \textbf{361} (2009), no.~4, 2109--2161.

\bibitem{miranda}
R.~Miranda, \emph{{The basic theory of elliptic surfaces. Notes of lectures.}},
  {Dottorato di Ricerca di Matematica, Universit\`a di Pisa, Dipartimento di
  Matematica. Pisa: ETS Editrice, vi, 106 p. }, 1989 (English).

\bibitem{Nikulin1}
V.V. Nikulin, \emph{Finite groups of automorphisms of {K}\"ahlerian surfaces of
  type {$K3$}}, Uspehi Mat. Nauk \textbf{31} (1976), no.~2(188), 223--224.

\bibitem{nikulinfactor}
\bysame, \emph{Factor groups of groups of automorphisms of hyperbolic forms
  with respect to subgroups generated by 2-reflections. {A}lgebrogeometric
  applications}, J. Soviet. Math. \textbf{22} (1983), 1401--1475.

\bibitem{matthias}
M.~Sch{\"u}tt, \emph{{$K3$} surfaces with non-symplectic automorphisms of
  2-power order}, J. Algebra \textbf{323} (2010), no.~1, 206--223.

\bibitem{SS}
M.~Sch{\"u}tt and T.~Shioda, \emph{Elliptic surfaces}, Algebraic geometry in
  {E}ast {A}sia---{S}eoul 2008, Adv. Stud. Pure Math., vol.~60, Math. Soc.
  Japan, Tokyo, 2010, pp.~51--160.

\bibitem{takiauto}
S.~Taki, \emph{Classification of non-symplectic automorphisms of order 3 on
  {K}3 surfaces}, Math. Nachr. \textbf{284} (2011), 124--135.

\bibitem{Taki}
\bysame, \emph{Classification of non-symplectic automorphisms on {$K3$}
  surfaces which act trivially on the {N}\'eron-{S}everi lattice}, J. Algebra
  \textbf{358} (2012), 16--26.

\bibitem{Taki32}
\bysame, \emph{On {O}guiso's {$K3$} surface}, J. Pure Appl. Algebra
  \textbf{218} (2014), no.~3, 391--394.

\end{thebibliography}
\end{document}